\documentclass{amsart}
\usepackage[T1]{fontenc}   

\newtheorem{theorem}{Theorem}[section]
\newtheorem{lemma}[theorem]{Lemma}
\newtheorem{corollary}[theorem]{Corollary}
\newtheorem{proposition}[theorem]{Proposition}
\theoremstyle{definition}
\newtheorem{definition}[theorem]{Definition}

\newtheorem{remark}[theorem]{Remark}
\newtheorem{remarks}[theorem]{Remarks}
\numberwithin{equation}{section}

\begin{document}


\renewcommand{\bf}{\bfseries}
\renewcommand{\sc}{\scshape}
\vspace{0.5in}

\newcommand\ssk{\smallskip}
 \newcommand\msk{\medskip}
\newcommand\bsk{\bigskip}
\newcommand\gp{\mathrm{gp}}
\newcommand\CC{\mathbb{C}}
\newcommand\RR{\mathbb{R}}
\newcommand\ZZ{\mathbb{Z}}
\newcommand\NN{\mathbb{N}}
\newcommand\QQ{\mathbb{Q}}
\newcommand\AAA{\mathbb{A}}
\newcommand\cc{\mathfrak{c}}
\newcommand\BB{\mathbb{B}}

\title[Transcendental Groups]{Transcendental Groups}

\author{Sidney A. Morris}
\address{School of Engineering, Information Technology and Physical Sciences, Federation University Australia, PO Box 663, Ballarat, Victoria, 3353,  Australia \&  Department of Mathematics and Statistics, La~Trobe University, Melbourne, Victoria, 3086, Australia}
\email{morris.sidney@gmail.com}
\thanks{Dedicated to Ralph Kopperman}



\begin{abstract} In this note we introduce the notion of a transcendental group, that is, a subgroup $G$ of the topological group $\mathbb{C}$  of all complex numbers such that every element of $G$ except $ 0$ is a transcendental number.  All such topological groups are separable metrizable zero-dimensional torsion-free abelian groups. Further,  each transcendental group is homeomorphic to a subspace of  $\mathbb{N}^{\aleph_0}$, where $\mathbb{N}$ denotes the discrete space of natural numbers. It is shown that  (i) each countably infinite transcendental  group is a member of one of three classes, where each class  has $\mathfrak{c}$ (the cardinality of the continuum) members -- the first class consists of those isomorphic as a topological group to the discrete group $\ZZ$ of integers, the second class consists of those isomorphic as a topological group to $\ZZ\times \ZZ$, and the third class consists of those homeomorphic to the topological space $\QQ$ of all rational numbers;  (ii) for each cardinal number $\aleph$ with $\aleph_0< \aleph\le \cc$, there exist $2^\aleph$  transcendental groups of cardinality $\aleph$ such that no two of the transcendental groups  are isomorphic as topological groups or even homeomorphic;
(iii) there exist $\mathfrak{c}$ countably infinite transcendental groups each of which is homeomorphic to $\QQ$ and algebraically isomorphic to a vector space over the field $\AAA$ of all algebraic numbers (and hence also over $\QQ$) of countably infinite dimension; (iv) $\RR$ has $2^\cc$ transcendental subgroups, each being a zero-dimensional metrizable torsion-free abelian group, such that  no two of the transcendental groups  are isomorphic as topological groups or even homeomorphic.
\end{abstract}

\maketitle

\section{\bf Introduction}
This paper initiates the study of transcendental groups which is an interesting combination of number theory, algebra, and topology.   We shall see that transcendental groups are a wonderfully  rich source of examples of zero-dimensional separable metrizable  topological groups. Transcendental   groups may also lead to a new way of looking at problems in transcendental number theory.   

\begin{remarks}\label{1.1}
We shall discuss four fields: $\CC$, the field of all complex numbers; $\RR$, the field of all real numbers; $\AAA$, the field of all algebraic numbers; and $\QQ$, the field of all rational numbers. Observe the following easily verified facts:
\begin{itemize}
\item[(i)] the fields $\CC$ and $\RR$ have cardinality $\cc$, the cardinalty of the continuum; 
\item[(ii)] the fields $\AAA$ and $\QQ$ have cardinality $\aleph_0$; 
\item[(iii)]  $\CC$  with its euclidean topology is isomorphic as a topological group to $\RR\times \RR$, where $\RR$ has its euclidean topology; 
\item[(iv)] each of these four fields has a natural topology; $\CC$  and $\RR$ have euclidean topologies, while $\AAA$ and $\QQ$ inherit a natural topology as a subspace of $\CC$;
\item[(v)] the topological group $\QQ$ is a dense subgroup of the topological group $\RR$ (that is, the closure, in the topological sense, of $\QQ$ is $\RR$);
\item[(vi)]   the topological group $\AAA$ is a dense subgroup of the topological group $\CC$; 
\item[(vii)]  $\CC\supset \AAA \supset\AAA\cap\RR \supset \QQ$, but $\AAA$ is not a subset of $\RR$;
\item[(viii)] the field $\CC$ is a vector space of dimension $\cc$  over $\AAA$ and it is also a vector space of dimension $\cc$ over $\QQ$; 
\item[(ix)] using the Axiom of Choice, we see that for any set of linearly independent vectors in a vector space $V$, there is another linearly independent set of  vectors in $V$ such that the union of the two linearly independent sets is a basis for the vector space $V$; this implies from (viii) that there exists a vector space $\BB$ over the field $\QQ$ such that  the vector space 
$\CC$  is  isomorphic as a vector space to the direct sum  of the vector spaces $\AAA$ and $\BB$ over $\QQ$;  that is, $\CC\cong \AAA\oplus \BB$; (This is an algebraic isomorphism and definitely not a topological group isomorphism since $\CC$ is a connected topological space while $\AAA$, being countable, is not a connected topological space.) 
\item[(x)] from (ix), the vector space $\BB$ has dimension $\cc$ over $\QQ$;
\item[(xi)] $\BB$, $\RR$, and $\CC$ are each a vector space of dimension $\cc$ over $\QQ$; so $\BB$, $\RR$, and $\CC$ are algebraically isomorphic as groups to each other and to a  restricted direct sum of $\cc$ copies of $\QQ$; 
\item[(xii)] $\AAA$ is  a vector space of countably infinite dimension over $\QQ$;
\item[(xiii)] $\BB$  and $\CC$ are each a vector space of dimension $\cc$ over $\AAA$;
\item[(xiv)] $\BB$ is a topological group which is algebraically isomorphic to both $\RR$ and $\CC$.
\end{itemize}
We shall focus on  $\mathcal{T}$, the topological space of all transcendental numbers, where $\mathcal{T}= \CC\setminus \AAA$ and has a natural topology as a subspace of $\CC$. The topology of $\mathcal{T}$ is separable, metrizable, and zero-dimensional. Also the cardinality of $\mathcal{T}$ is $\cc$.\qed\end{remarks}

\begin{definition}\label{1.2} A topological group $G$ is said to be be a \emph{transcendental group} if  it is a subgroup  of the topological group $\mathbb{C}$  of all complex numbers 
such that every element of $G$ except $ 0$ is a transcendental number.   \qed \end{definition}

As transcendental groups are topological subgroups of the separable metrizable group $\CC$, each is separable and metrizable and has cardinality not greater than $\cc$.    As $\CC$ is a torsion-free  abelian group, every transcendental group is an infinite torsion-free abelian group. Since transcendental groups  are subspaces of the space $\mathcal{T}$, they are  zero-dimensional. This is summarized in Proposition~\ref{1.3}.

\begin{proposition}\label{1.3} Every transcendental group is an infinite separable metrizable zero-dimensional torsion-free abelian group of cardinality not greater than $\cc$. \qed
\end{proposition}

\begin{proposition}\label{1.4} The topological group $\BB$ introduced in Remarks~\ref{1.1} is a transcendental group. Further, every transcendental group is algebraically isomorphic to a subgroup of $\BB$.
\end{proposition}

\begin{proof} Clearly from the definition of $\BB$ in Remarks~\ref{1.1}, $\BB$ is a subset of $\mathcal{T}$ and so it is a transcendental group. Let $p$ be the projection homomorphism of $\CC$ onto $\BB$.  Let the  the subgroup $T$ of $\CC$ be  a transcendental group. 
If $t\in T,\ t\ne 0$, then $p(t)\ne 0$ since  otherwise  $t\in \AAA$, which is false as $t$ is transcendental.  So $p$ is a one-to-one homomorphism of $T$ onto the subgroup $p(T)$ of  $\BB$. Thus  $T$ is algebraically isomorphic to a subgroup of $\BB$, and the proposition is proved.  
\end{proof}

\section{\bf Discrete Transcendental Groups}

If $S$ is a subset of a group $G$, then $\gp(S)$ denotes the subgroup of $G$ generated algebraically by the set $S$. If $S$ is the singleton set $\{g\}$, then this group is denoted by $\gp\{g\}$ and equals   $\{ng:n\in\ZZ\}$.

 \ssk
 Our first proposition is obvious.

\begin{proposition}\label{2.1} Let $t$ be any transcendental number in $\CC$. Then $\gp\{t\}$ is a discrete countably infinite  transcendental group and  is isomorphic as a topological group to the discrete group $\mathbb{Z}$ of all integers. \qed
\end{proposition}

 \begin{corollary}\label{2.2} There exist  $\mathfrak{c}$ distinct  transcendental groups each of which is topologically isomorphic to $\ZZ$.
 \end{corollary}

\begin{proof} If $\gp\{t_1\}= \gp\{t_2\}$, for positive  transcendental numbers $t_1$ and $t_2$, then $t_1=n_1t_2$ and $t_2=n_2t_1$, for $n_1,n_2\in \NN$. But then $ t_1=n_1n_2t_1$, which implies $n_1=n_2=1$ and so $t_1=t_2$. Thus  distinct positive transcendental numbers generate distinct transcendental groups each of which is  topologically isomorphic to $\ZZ$. The result now immediately follows from the fact that there are  $\mathfrak{c}$ distinct positive transcendental numbers
\end{proof}

Let us now turn to groups generated by two (unequal) transcendental numbers $t_1$ and $t_2$. If $t_2=t_1+1$, then $\gp\{t_1,t_2\}$ contains $t_2-t_1=1$. So $\gp\{t_1,t_2\}$ is not a transcendental group. But we are very close to open questions. For example, while the Euler number e and the number $\pi$ are transcendental numbers, it is not known if either $\pi+\,$e or $\pi-\,$e is transcendental (although one of them must obviously be transcendental).

Much modern transcendental number theory centres on linear independence. (See  \cite{murty}.) In this context we have our next proposition.

\begin{proposition}\label{2.3} Let $a$ and $b$ be transcendental numbers with $a,b\ne0$.  Then $\gp\{a,b\}$  is a cyclic transcendental group if and only if $a$ and $b$ are linearly dependent over $\QQ$.
\end{proposition}

\begin{proof} If $\gp\{a,b\}$ is a cyclic group, then there exists  $c\in\gp\{a,b\}$ such that $a=mc$ and $b=nc$, for some $m,n\in \ZZ\setminus\{0\}$. So  $na-mb=0$. So $a$ and $b$ are linearly dependent over $\QQ$.

Now consider the case that $a$ and $b$ are linearly dependent over $\QQ$; that is, there exist $m_1,m_2,n_1,n_2\in \ZZ$, with $m_1,m_2,n_1,n_2\ne0$, such that $$\frac{m_1}{n_1}a+\frac{m_2}{n_2}b=0$$
Consider the element $d\in \CC$, defined by $d=\dfrac{a}{m_2n_1}$. The number  $d$ is  transcendental  and $a\in \gp\{d\}$. As $b=-\dfrac{ m_1n_2}{m_2n_1}a=-m_1n_2d$  we also have $b\in\gp\{d\}$. Thus $\gp\{a,b\}$ is a subgroup of the cyclic transcendental group $\gp\{d\}$. Hence $\gp\{a,b\}$ is a cyclic transcendental group.
\end{proof}

Proposition~\ref{2.3} tells us, for example, that the numbers e and $\pi$ are linearly independent over $\QQ$ if and only if $\gp\{\mathrm{e}, \pi\}$ is not a cyclic transcendental  group.   If e$\, +\,\pi$ is an algebraic number, then $\gp\{\mathrm{e},\pi\}$ is not a transcendental group. So e and $\pi$ must be linearly independent over $\QQ$. Note that if Schanuel's Conjecture \cite{Lang, murty} is true, then e and $\pi$ are indeed linearly independent over $\QQ$.

\begin{proposition}\label{2.4} If $G$ is a discrete transcendental group and $G\ne \{0\}$, then  $G$ is topologically isomorphic to either $\ZZ$ or $\ZZ\times\ZZ$. 
\end{proposition}

\begin{proof} By Theorem 6 of \cite{morris} or Theorem A.12 of \cite{compbook}, every nontrivial closed subgroup of $\RR^n$, for $n\in \NN$, is topologically isomorphic to $\RR^a\times \ZZ^b$, where $a,b\in \NN\cup\{0\}$ and  $a+b\le n$.  As $\CC$ is topologically isomorphic to $\RR^2$, every discrete  subgroup of $\CC$ is topologically isomorphic to $\ZZ$ or $\ZZ\times\ZZ$. 
\end{proof}

We shall see in Theorem~\ref{4.5}  that, if $i$ as usual denotes an imaginary number who square is $-1$, the transcendental group $\gp\{\mathrm{e},\mathrm{e}i\}$ is topologically isomorphic to the discrete $\ZZ\times\ZZ$.

\section{\bf Counting   Transcendental Groups}

Firstly, in this section, we identify $\cc$ concrete non-cyclic countably infinite transcendental groups. We begin by recalling  the following theorem:

\msk
\noindent\bf Generalized Lindemann Theorem. {\rm \cite[Theorem 9.1]{Niven} }{\it For any distinct algebraic numbers $\alpha_1,\alpha_2,\dots,\alpha_m$, $m\in \NN$, the values $e^{\alpha_1},e^{\alpha_2},\dots,e^{\alpha_m}$ are linearly independent over the field $\AAA$ of algebraic numbers.} \qed

\begin{theorem}\label{3.1} Let $\Omega$ be any finite or infinite set of distinct algebraic numbers such that $0\notin \Omega$. Then the set $\Gamma=\{e^a: a\in \Omega\}$ generates the transcendental group  $\gp(\Gamma)$.
\end{theorem}

\begin{proof}
Let   $\alpha_2,\alpha_3,\dots,\alpha_m    \in \Omega$, $m\in \NN$, and put $\alpha_1=0$.  The Generalized Lindemann Theorem implies that $e^{\alpha_1},e^{\alpha_2},\dots,e^{\alpha_m}$ are linearly independent over the field $\AAA$ of algebraic numbers; that is,  for $a_1,a_2,\dots,a_m\in \AAA$ with  $a_1,a_2,\dots,a_m\ne0$,
$$a_1e^{\alpha_1}+a_2e^{\alpha_2}+\dots+a_me^{\alpha_m}\ne0.$$
As $a_1e^{\alpha_1}=a_1\in \AAA$, this says 
$$a_2e^{\alpha_2}+\dots+a_me^{\alpha_m}\notin \AAA.$$
Thus $\Gamma$ is a transcendental group.
\end{proof}

\begin{remark}\label{3.2}
Noting that there are $\cc$ distinct sets $\Omega$ satisfying the conditions of Theorem~\ref{3.1}, we obtain $\cc$ distinct countably infinite  transcendental groups $\gp(\Gamma)$ each of which is algebraically isomorphic to a vector space over $\AAA$ (and hence also over $\QQ$) of countably infinite dimension.\qed
\end{remark}

\begin{theorem}\label{3.3} Let $S=\{\alpha_i:i\in I\}$ for some index set $I$ such that  each $\alpha_i$ is an algebraic number with $\alpha_i\ne 0$. Put $T=\{\log(\alpha_i):\alpha_i\in S\}$. Then $\gp(T)$ is a  transcendental group.
\end{theorem}

\begin{proof}
Let $g\in \gp(T)$, $g\ne 0$. Then for each $i=1,2,\dots,n$ with $n\in\NN$,
\begin{align*} g=&m_1\log\{\alpha_1\}+ m_2\log\{\alpha_2\} +\dots+m_n\log\{\alpha_n\}, \ m_i\in \ZZ\setminus\{0\}\\
=&\log(\alpha_1^{m_1}\alpha_2^{m_2}\dots\alpha_n^{m_n}).\end{align*} So
$\mathrm{e}^g=\alpha_1^{m_1}\alpha_2^{m_2}\dots\alpha_n^{m_n}$. 

Suppose $g$ is an algebraic number. Then by Theorem 9.11 of \cite{Niven}, $\mathrm{e}^g$ is a transcendental number, while $ \alpha_1^{m_1}\alpha_2^{m_2}\dots\alpha_n^{m_n}$ is an algebraic number, which is a contradiction. So $g$ is a transcendental  number, and $\gp(T)$ is a transcendental group.
\end{proof}

\begin{theorem}\label{3.4} There exist $2^\cc$  transcendental groups.
\end{theorem}

\begin{proof}  Consider $\CC$ as a vector space over $\AAA$ of dimension $\cc$. Let $R$ be a basis for this vector space such that $1\in R$. Put 
$T= \{x\in R: x\notin \AAA\}.$
For each subset $S'$ of $T$, let $S=S'\cup \{1\}$. Then $S$ is linearly independent over $\AAA$. 

Let $g\in \gp(S'), g\ne0$ and suppose $g$ is not a transcendental number. 
Then $g=m_1s_1+m_2s_2+\dots+m_ns_n$, where $s_1,s_2,\dots,s_n\in S'$ and $m_1,m_2,\dots,m_n\in \ZZ\setminus \{0\}$. Thus
$$g=m_1s_1+m_2s_2+\dots+m_ns_n= a\in \AAA.$$
So $m_1s_1+m_2s_2+\dots+m_ns_n+(-a).1=0$. But this contradicts the fact that $s_1,s_2,\dots,s_n,1\in S$ and so are linearly independent over $\AAA$. Hence  $g$ is indeed a transcendental number. Therefore $\gp(S')$ is a transcendental group. As there are $2^\cc$  different such $S'$ each generating a different group, the theorem is proved.
\end{proof}

\rm As a consequence of the proof of Theorem~\ref{3.4} we have the following theorem.

\begin{theorem}\label{3.5}
For each cardinal number $\aleph$ with $\aleph_0\le \aleph\le \cc$, there exist $2^\aleph$  transcendental groups of cardinality $\aleph$. \qed
\end{theorem}

If in the proof of Theorem~\ref{3.4} we replace $\CC$ by $\RR$ and $\AAA$ by $\AAA\cap \RR$, then we obtain the following theorem:

\begin{theorem}\label{3.6}For each cardinal number $\aleph$ with $\aleph_0\le \aleph\le \cc$, there exist $2^\aleph$  transcendental groups, each of which is a topological subgroup of $\RR$ and has cardinality $\aleph$. \qed
\end{theorem}

\begin{remark}\label{3.7}  Corollary 1.2 of  \cite{Comfort}  states that if $G$ is an abelian group of cardinality $\aleph>\aleph_0$, then $G$ has $2^\aleph$  subgroups. If one knew that $G$ has  as a subgroup a transcendental group of cardinality $\aleph$, then one could deduce Theorem~\ref{3.5} except for the case $\aleph=\aleph_0$. However, the proof in \cite{Comfort} is not shorter than the one presented here.
\qed
\end{remark}

\section{\bf The Topology of Transcendental Groups}\label{4}

\begin{remark}\label{4.1}
Jan van Mill drew my attention to    Theorem 1.9.6, Theorem 1.9.8, and Corollary 1.9.9 of  \cite{vanMill} where it is proved that
\begin{itemize}
\item[(i)] the  space $\QQ$ of all rational numbers
 up to homeomorphism is the unique nonempty countable separable metrizable space without isolated points. 
 \item[(ii)] the  space $\mathbb{P}$ of all irrational numbers up to homeomorphism is  the unique nonempty separable metrizable topologically complete nowhere locally compact zero-dimensional space.  
 \item[(iii)] $\mathbb{P}$ is homeomorphic to $\NN^{\aleph_0}$.  
\end{itemize}
We have already noted  that all transcendental groups are separable, metrizable and zero-dimensional.   \qed\end{remark}

As immediate corollaries of the observations in Remark~\ref{4.1} we have:

\begin{corollary}\label{4.2} The topological space $\AAA$  of all algebraic numbers  is homeomorphic to $\QQ$.\qed\end{corollary}

\begin{corollary}\label{4.3} Every countably infinite separable metrizable topological group is either discrete  or homeomorphic to $\QQ$.\qed\end{corollary}

\begin{lemma}\label{4.4}
Let $a$ be any transcendental number which is also a nonzero real number. Then $\gp\{a,ai\}$ is a transcendental group isomorphic as a topological group to $\ZZ\times \ZZ$. 
\end{lemma}

\begin{proof}
By Remarks~\ref{4.1}(i), $\gp\{a,ai\}$ is either a discrete group or has $0$ as a nonisolated point. But for $z\in \gp\{a,ai\}$, $z\ne 0$, for $m,n\in \ZZ$, not both zero, 
$$|z|=|ma+nai| = \sqrt{(m^2+n^2)a^2}\ge |a|.$$ So $0$ is indeed an isolated point. Thus $\gp\{a,ai\}$ is discrete.
\end{proof}

\begin{theorem}\label{4.5} There exist $\cc$ countably infinite transcendental groups of the form $\gp\{a,ai\}$ where each is  isomorphic as a topological group to $\ZZ\times \ZZ$. \qed
\end{theorem}

Using a similar argument used to prove those  statements in Remark~\ref{4.1}, one readily obtains:

\begin{corollary}\label{4.6} The topological space of transcendental numbers is homeomorphic to both $\mathbb{P}$ and $\NN^{\aleph_0}$.\qed\end{corollary}

\begin{remark}\label{4.7} It follows from Corollary~\ref{4.6} that every transcendental group is homeomorphic to a subspace of $\NN^{\aleph_0}$.\qed
\end{remark}

\begin{remark}\label{4.8} From Corollary~\ref{4.3} and Theorem~\ref{3.3} we see that if $S=\alpha_n=1+\frac{1}{n}$, for each $n\in \NN$, and $T=\{\log(\alpha_n):n\in \NN\}$, then 0 is a limit point of $\gp(T)$. So the topological group $\gp(T)$ has no isolated points and therefore is homeomorphic to $\QQ$.
\qed\end{remark}

\begin{theorem}\label{4.9} There exist $\cc$ transcendental groups homeomorphic to $\QQ$.
\end{theorem}

\begin{proof}
Let $S'$ be any subset of $\AAA$ such that $\frac{1}{n}\notin S'$, for all $ n\in \NN$. Put $S=S'\cup \{\frac{1}{n}: n\in \NN\}$. Define $T_{S} =\{e^\alpha: \alpha\in S\}$. As observed in Theorem~\ref{3.1}, $\gp(T_{S})$ is a transcendental group. Clearly $e^\frac{1}{n}\to 1$ as $n\to \infty$. So the countably infinite group $\gp(T_{S})$ is not discrete.  By Corollary~\ref{4.3}, $\gp(T_{S})$ is therefore homeomorphic to $\QQ$. As there are $\cc$  different possible choices of $S'$, the theorem is proved.
\end{proof}

\begin{theorem}\label{4.10} Each countably infinite transcendental  group is a member of one of three classes, where each class  has $\mathfrak{c}$  members -- the first class consists of those isomorphic as a topological group to the discrete group $\ZZ$ of integers, the second class consists of those isomorphic as a topological group to $\ZZ\times \ZZ$, and the third class consists of those homeomorphic to the topological space $\QQ$ of all rational numbers.
\end{theorem}

\begin{proof} The theorem is an immediate consequence of
 Corollary 2.2, Proposition 2.4, Corollary 4.3, Theorem 4.5, and Theorem 4.9.  
\end{proof}

\begin{remark}\label{4.11} Jan van Mill has pointed out to me a beautiful consequence of Thorem~\ref{3.4}. The Laverentieff Theorem,  Theorem A8.5 of \cite{vanMill}, implies that there are at most $\cc$ subspaces of $\CC$ which are homeomorphic. 
So from Theorem~\ref{3.4} there are $2^\cc$ transcendental groups no two of which are homeomorphic. 

Of course this trivially has the consequence there are $2^\cc$ transcendental groups no two of which are isomorphic as topological groups. 

So the topological group $\BB$, introduced in Remarks~\ref{1.1} which is algebraically isomorphic to $\RR$, has $2^\cc$ subgroups no two of which are  isomorphic as topological groups.  
\qed\end{remark}

Noting Remark~\ref{4.11} and Theorem~\ref{3.6} we obtain:

\begin{theorem}\label{4.12} For each cardinal $\aleph$ with $\aleph_0< \aleph\le\cc$, the topological groups $\CC$ and $\RR$ each have $2^\aleph$ transcendental subgroups no two of which are isomorphic as topological groups or even homeomorphic.
\qed
\end{theorem}

Our next corollary is clear since each Banach space has a subgroup isomorphic as a topological group   to $\RR$. The importance of closed totally disconnected subgroups of Banach is well known, see \cite{Ancel, Sidney}.

\begin{corollary}\label{4.13} Let   the separable Banach space $E$ be finite or infinite dimensional.  For each cardinal $\aleph$ with $\aleph_0<  \aleph\le\cc$,  $E$ has 
$2^\aleph$ transcendental subgroups no two of which are isomorphic as topological groups or even homeomorphic.\qed \end{corollary}

\noindent \bf Acknowledgment.  \rm Special thanks go to Karl Heinrich Hofmann and Jan van Mill for their insightful comments on an earlier draft of this paper which resulted in a much improved paper. Thanks also to Wayne Lewis for suggesting possibilities for future research on this topic.

\ssk

\noindent\bf Dedication. \rm This paper is dedicated to  Ralph Kopperman who was a friend and coauthor. He visited me in  Australia a number  of times and I visited him in New York. I enjoyed doing research with him. We discussed coauthoring a book but we were both too strong-minded for that to be a success. He worked tirelessly for the Summer Topology Conference.  My fondest memory is of his visiting me in Armidale, New South Wales a regional Australian city of 25,000 people hundreds of miles from any large city and his thoroughly enjoying the wonderful Southern Hemisphere night sky. He was particularly impressed with the  Clouds of Magellan (two dwarf galaxies visible in the Southern  Hemisphere sky).

\bibliographystyle{plain}

\end{document}